\documentclass{amsart}
\usepackage{amsmath, amstext, amsbsy, amssymb}

\hoffset \voffset \oddsidemargin=55pt \evensidemargin=55pt
\numberwithin{equation}{section}
\def\beq{\begin{eqnarray}}
\def\eeq{\end{eqnarray}}
\def\beqs{\begin{eqnarray*}}
\def\eeqs{\end{eqnarray*}}

\def\mz{{\mathbb Z}}
\def\mr{{\mathbb R}}

\def\ind{\hbox{\rm ind}}
\def\ord{\hbox{\rm ord}}

\newfont{\df}{eufm10}

\title[Index-conjecture of length four sequences]
{On the index-conjecture of length four minimal zero-sum sequences II}
\thanks{$^\dag$the corresponding author's email: xialimeng@ujs.edu.cn}
\thanks{Supported by the  NNSF of China (Grant No. 11001110, 11271131)}
\author[C.-X. Shen]{Caixia   Shen}
\author[L.-M. Xia]{Li-meng   Xia$^\dag$}

\date{}
\begin{document}
\maketitle
\centerline{Faculty of Science, Jiangsu University}
\centerline{Zhenjiang, 212013, Jiangsu Province, P.R. China}

\def\abstractname{ABSTRACT}
\begin{abstract}
Let $G$ be a finite cyclic group. Every sequence $S$ over $G$ can be written in the form
$S=(n_1g)\cdot...\cdot(n_lg)$ where $g\in G$ and $n_1,\cdots,n_l\in[1,{\hbox{\rm ord}}(g)]$, and the index $\ind S$ of $S$ is defined to be the minimum of $(n_1+\cdots+n_l)/\hbox{\rm ord}(g)$ over all possible $g\in G$ such that $\langle g\rangle=G$. A conjecture says that if $G$ is finite such that $\gcd(|G|,6)=1$, then $\ind(S)=1$ for every minimal zero-sum sequence $S$. In this paper, we prove that the conjecture holds if $S$ is reduced and the (A1) condition is satisfied(see [19]).


\end{abstract}

\newtheorem{theo}{Theorem}[section]
\newtheorem{theorem}[theo]{Theorem}
\newtheorem{defi}[theo]{Definition}
\newtheorem{conj}[theo]{Conjecture}
\newtheorem{lemma}[theo]{Lemma}
\newtheorem{coro}[theo]{Corollary}
\newtheorem{proposition}[theo]{Proposition}
\newtheorem{remark}[theo]{Remark}

\setcounter{section}{0}

\section{Introduction}

Throughout the paper, let $G$ be an additively written finite cyclic group of order $|G| = n$. By
a sequence over $G$ we mean a finite sequence of terms from $G$ which is unordered and repetition
of terms is allowed. We view sequences over $G$ as elements of the free abelian monoid $\mathcal{F}(G)$
and use multiplicative notation. Thus a sequence $S$ of length $|S| = k$ is written in the form
$S = (n_1g)\cdot...\cdot(n_kg)$, where $n_1,\cdots,n_k\in{\mathbb N}$ and $g\in G$. We call $S$ a {\it zero-sum sequence} if $\sum^k_{j=1}n_jg = 0$. If $S$ is a zero-sum sequence, but no proper nontrivial subsequence of $S$ has sum zero, then $S$ is called a {\it minimal zero-sum sequence}. Recall that the index of a sequence $S$ over $G$
is defined as follows.

\begin{defi}
For a sequence over $G$
\beqs S=(n_1g)\cdot...\cdot(n_kg), &&\hbox{where}\;1\leq n_1,\cdots,n_k\leq n,\eeqs
the index of $S$ is defined  by $\ind(S)=\min\{\|S\|_g|g\in G \hbox{~with~}\langle g\rangle=G\}$, where
\beq \|S\|_g=\frac{n_1+\cdots+n_k}{\ord(g)}.\eeq
\end{defi}
Clearly, $S$ has sum zero if and only if $\ind(S)$ is an integer.

\begin{conj}
Let $G$ be a finite cyclic group such that $\gcd(|G|,6)=1$. Then every minimal zero-sum sequence $S$ over $G$ of length $|S|=4$ has $\ind(S)=1$.
\end{conj}

The index of a sequence is a crucial invariant in the investigation of (minimal) zero-sum
sequences (resp. of zero-sum free sequences) over cyclic groups. It was first addressed by
Kleitman-Lemke (in the conjecture [9, page 344]), used as a key tool by Geroldinger ([6, page736]), and then investigated by Gao [3] in a systematical way. Since then it has received a great
deal of attention (see for example [1, 2, 4, 7, 10, 11, 12, 13, 14, 15, 16, 17, 18]). A main focus of the investigation of index is to determine minimal zero-sum sequences of index 1. If $S$ is a minimal zero-sum sequence of length $|S|$ such that $|S|\leq3$ or $|S|\geq\lfloor \frac{n}2\rfloor+2$, then $\ind(S)=1$ (see [1, 14, 16]). In contrast to that, it was shown that for each $k$ with $5\leq k\leq \lfloor \frac{n}2\rfloor+1$, there is a minimal zero-sum subsequence $T$ of length $|T| = k$ with $\ind(T)\geq 2$ ([13, 15]) and that the same is true for $k = 4$ and $\gcd(n, 6)\not= 1$ ([13]). The left case leads to the above conjecture.

In [12], it was proved that Conjecture 1.2 holds true if $n$ is a prime power. In [11], it was proved that Conjecture 1.2 holds for $n=p_1^\alpha\cdot p_2^\beta$, $(p_1\not=p_2)$, and at
least one $n_i$ co-prime to $|G|$.

In [19], it was proved that Conjecture 1.2 holds if the sequence $S$ is reduced and at least one $n_i$ co-prime to $|G|$.

By the result of [19], a minimal zero-sum sequence $S=(x_1g)\cdot(x_2g)\cdot(x_3g)\cdot(x_4g)$ over $G$ is reduced then $|G|$ has at most two prime factors or one of the following holds:

(A1) $\{\gcd(x_i,n)|i=1,2,3,4\}=\{p_1,p_2,p_1p_3,p_2p_3\}$;

(A2) $\{\gcd(x_i,n)|i=1,2,3,4\}=\{1,p_1,p_2,p_1p_2\}$;

(A3) $\gcd(x_i,n)=1$ for $i=1,2,3,4$;

(A4) $\gcd(x_1,n)=1,\gcd(x_2,n)=p_1p_2,\gcd(x_3,n)=p_1p_3,\gcd(x_4,n)=p_2p_3$.

In this paper, we give the affirmative proof under assumption (A1), and our main result can be stated by the following theorem:

\begin{theo}
Let $G=\langle g\rangle$ be a  finite cyclic group such that $|G|=p_1p_2p_3$ and $\gcd(n,6)=1$. If $S=(x_1g,x_2g,x_3g, x_4g)$ is a minimal zero-sum  sequence over $G$ such that
\beqs \{\gcd(n,x_i)|i=1,2,3,4\}=\{p_1,p_2,p_1p_3,p_2p_3\}.\eeqs
Then  $\ind(S)=1$.
\end{theo}

It was mentioned in [13] that Conjecture 1.2 was confirmed computationally if $n\leq1000$.
Hence, throughout the paper, we always assume that $n>1000$.

\section{Preliminaries and renumbering the sequence}

Throughout, let G be a cyclic group of order $|G| = n > 1000$. Given real numbers $a, b\in\mr$,
we use $[a, b] = \{x\in\mz|a\leq x\leq b\}$ to denote the set of integers between $a$ and $b$, and similarly, set $[a, b) = \{x\in\mz|a\leq x < b\}$. For $x\in\mz$,
we denote by $|x|_n\in[1, n]$ the integer congruent to $x$ modulo $n$. Suppose that $n$ has a prime
decomposition $n = p^\alpha q^\beta$. Let $S = (x_1g)\cdot...\cdot(x_4g)$ be a minimal zero-sum sequence over $G$
such that $\ord(g) = n = |G|$ and $1\leq x_1, x_2, x_3, x_4\leq n-1$. Then $x_1 + x_2 + x_3 + x_4 = \nu n$, where
$1\leq\nu\leq3$.

Let $S$ be the sequence as described in Theorem 1.3.  Similar to Remark 2.1 of [11], we may always assume that
$x_1=e, e+x_2+x_3+x_4=2n$ and $e<x_2<\frac{n}{2}<x_3\leq x_4<n-e$. Let $c=x_2,b=n-x_3,a=n-x_4$, then it is easy to show that the following proposition implies Theorem 1.3.

\begin{proposition}
Let $n=p_1p_2p_3$, where $p_1, p_2, p_3$ are three different primes, and $\gcd(n,6)=1$. Let $S= (g)\cdot(cg)\cdot((n-b)g)\cdot((n-a)g)$ be a minimal zero-sum sequence over $G$ such that $\ord(g)=|G|=n$, and
\beqs \{\gcd(n,e),\gcd(n,c),\gcd(n,b),\gcd(n,a)\}=\{p_1,p_2,p_1p_3,p_2p_3\},\eeqs
where $e+c=a+b$. Then $\ind(S)=1$.
\end{proposition}

{\noindent \emph{Notice that:}} for convenience, we list two sufficient conditions introduced in Remark 2.1 of [11].

(1) If there exists positive integer $m$ such that $\gcd(n,m)=1$ and $|mx_1|_n+|mx_2|_n+|mx_3|_n+|mx_4|_n=3n$, then $\ind(S)=1$.

(2) If there exists positive integer $m$ such that $\gcd(n,m)=1$ and at most one $|mx_i|_n\in\left[1,\frac{n}{2}\right]$ (or, similarly, at most one $|mx_i|_n\in\left[\frac{n}{2},n\right]$), then $\ind(S)=1$.

\begin{lemma}
Proposition 2.1 holds if one of the following conditions holds:

(1) There exist positive integers $k,m$ such that $\frac{kn}{c}\leq m\leq\frac{kn}{b}$, $\gcd(m,n)=1$, $1\leq k\leq b$, and $ma<n$.

(2) There exists a positive integer $M\in[1,\frac{n}{2e}]$ such that $\gcd(M,n)=1$ and at least two of the following inequalities hold:
$$|Ma|_n>\frac{n}{2}, |Mb|_n>\frac{n}{2}, |Mc|_n<\frac{n}{2}.$$
\end{lemma}

\begin{lemma}
If there exist integers $k$ and $m$ such that $\frac{kn}{c}\leq m\leq\frac{kn}{b}$, $\gcd(m,n)=1$, $1\leq k\leq b$, and $a\leq\frac{b}{k}$, then Proposition 3.1 holds.
\end{lemma}

From now on, we assume that $s=\lfloor\frac{b}{a}\rfloor$. Then we have $1\leq s\leq\frac{b}{a}<s+1$. Since $b\leq\frac{n}{2}$, we have $\frac{n}{2b}=\frac{(2s-t)n}{2b}-\frac{(2s-t-1)n}{2b}>1$, and then $[\frac{(2s-t-1)n}{2b},\frac{(2s-t)n}{2b}]$ contains at least one integer for every $t\in[0,s-1]$.

\begin{lemma}
Suppose that $a>2e$, $s\geq2$ and $[\frac{(2s-2t-1)n}{2b},\frac{(s-t)n}{b}]$ contains an integer co-prime to $n$ for some $t\in[0,\cdots,\lfloor\frac{s}2\rfloor-1]$. Then Proposition 2.1 holds.
\end{lemma}

For the proof of Lemma 3.3, Lemma 3.4 and Lemma 3.5, one is referred to the proof of Lemma 2.3-2.5 in \cite{LP2}, and we omit it here.

Let $\Omega$ denote the set of those integers: $x\in \Omega$ if and only if $x\in [\frac{(2s-t-1)n}{2b},\frac{(s-t)n}{b}]$ for some $t\in[0,\lfloor\frac{s}2\rfloor-1]$. By Lemma 3.5, we also assume that

$(B)$: $[\frac{(2s-2t-1)n}{2b},\frac{(s-t)n}{b}]$ contains no integers co-prime to $n$ for every $t\in[0,\lfloor\frac{s}2\rfloor-1]$.

\begin{lemma}
Suppose that $a>2e$, $s\geq2$ and $[\frac{(2s-2t-1)n}{2b},\frac{(s-t)n}{b}]$ contains no integers co-prime to $n$ for every $t\in[0,\lfloor\frac{s}2\rfloor-1]$. Then $[\frac{(2s-t-1)n}{2b},\frac{(s-t)n}{b}]$ contains at most 3 integers for every $t\in[0,\lfloor\frac{s}2\rfloor-1]$. Hence $\frac{n}{2b}<4$.
\end{lemma}

\begin{lemma}
Suppose that $a>2e$, $s\geq4$ and $[\frac{(2s-2t-1)n}{2b},\frac{(s-t)n}{b}]$ contains no integers co-prime to $n$ for every $t\in[0,\lfloor\frac{s}2\rfloor-1]$. Then $[\frac{(2s-2t-1)n}{2b},\frac{(s-t)n}{b}]$ contains at most two integers for every $t\in[0,\lfloor\frac{s}2\rfloor-1]$ and $\frac{n}{2b}<3$.
\end{lemma}

\begin{lemma}
Suppose that $a>2e$ and $s\geq 6$, then there exists $t_1\in\{0,\lfloor\frac{s}2\rfloor-1\}$ such that   $[\frac{(2s-t_1-1)n}{2b},\frac{(2s-t_1)n}{2b}]$ contains exactly one integer and $\frac{n}{2b}<2$.
\end{lemma}

\begin{lemma}
Suppose that $a>2e$ and $s\geq 8$, then $[\frac{(2s-2t-1)n}{2b},\frac{(2s-t)n}{b}]$ contains exactly one integer for every $t\in[0,\lfloor\frac{s}2\rfloor-1]$.
\end{lemma}

\begin{lemma}
Under assumption $a>2e$, we have $s\leq9$.
\end{lemma}

For the proof of Lemma 2.2-2.9 and  more details, one is referred to [20], Li and Peng's paper [11] is also recommended.

Out of question, we can assume that $e=\min\{p_1,p_2\}$, without less of generality, let $p_1<p_2$, then $e=p_1$.

\begin{lemma}
If $a<4e$, then $p_3<p_1=e<p_2$ and $a=kp_2$ for some $k\in[1,3]$.
\end{lemma}
\begin{proof}
Since $p_2>p_1$ and $p_2p_3>p_1p_3\geq5p_1$, it must hold that $a=kp_2$ for some $k\in[1,3]$. Hence we only need prove the case $p_3>p_1$.

If $p_3>p_1$ and $a=3p_2$, it holds that $p_3|(c-b)=(a-e)=3p_2-p_1\in2\mz$. If $3p_2-p_1\geq4p_3$, we have $a\geq 4p_3+p_1>5p_1=5e$, a contradiction. Then $3p_2-p_1=2p_3$ and $p_3>p_2$. Simply computing shows that $p_1\geq29$ and $p_3\geq41$. Then $a=3p_2>2p_1=2e$ and $\frac{b}{a}\geq\frac{p_1p_3}{3p_2}>\frac{p_3}{4}>10$, which contradicts to the results of Lemma 2.9.

If $p_3>p_1$ and $a=p_2$, it holds that $p_3|(c-b)=(a-e)=p_2-p_1\in2\mz$. If $p_2-p_1\geq 4p_3$, we have $a>5e$, a contradiction. Then $p_2-p_1=2p_3$. Applying Lemma 2.9, similar to above, we have
$b=2p_1p_3, p_1=11, p_3=13, p_2=37$, or $b=p_1p_3$. If  $p_1=11, p_3=13, p_2=37$ and $b=2p_1p_3$, then $p_2p_3|c=b+a-e=2p_1p_3+p_2-p_1=312<481=p_2p_3$, a contradiction. If $b=p_1p_3$, then $p_2p_3|c=b+a-e=p_1p_3+p_2-p_1\leq(p_2-4)p_3+p_2-p_1<p_2p_3$, which is a contradiction.

If $p_3>p_1$ and $a=2p_2$, it holds that $p_3|(c-b)=(a-e)=2p_2-p_1\in2\mz+1$. If $2p_2-p_1\geq3p_3$, we have $a>4e$, a contradiction. Then $2p_2-p_1=p_3$. Applying Lemma 2.9, similar to above, we have $b=p_1p_3$ and $p_1=7, p_2=13, p_3=19$ or $p_1=11, p_2=17, p_3=23$. Since $p_2p_3|c=b+a-e=p_1(p_3-1)+p_2<p_3p_2$, we obtain a contradiction.
\end{proof}

\begin{lemma}
If $2e<a<4e$, then $\ind(S)=1$.
\end{lemma}
\begin{proof}
By Lemma 2.10, it holds that $a=kp_2$ for some $k\in[1,3]$. We distinguish three
cases according to the value of $k$.

{\bf Case 1.} $k=1$.

{\it Subcase 1.1.} $p_2|c$.

If $c=p_2p_3$, then $p_1|c-a=p_2(p_3-1)$, which implies $p_1|(p_3-1)<p_1$, a contradiction.

If $c=3p_2p_3$, then $\frac{b}{a}=\frac{c+e-a}{a}>3p_3-1\geq 14>10$, a contradiction.

If $c=2p_2p_3$, then $p_1|c-a=p_2(2p_3-1)$, which implies that $p_1=2p_3-1$. If $p_3\geq7$, we have $\frac{b}{a}=\frac{c+e-a}{a}\geq13>10$, a contradiction. If $p_3=5$, then $p_1|9$, a contradiction.

{\it Subcase 1.2.} $p_2|b$.

If $b=p_2p_3$, then $p_1|b+a=p_2(p_3+1)$, which implies $p_1=p_3+1<p_1$, a contradiction.

If $b\geq2p_2p_3$, then $\frac{b}{a}\geq10$, a contradiction.

{\bf Case 2.} $k=2$.

{\it Subcase 2.1.} $p_2|c$.

If $c=p_2p_3$, then $p_1|c-a=p_2(p_3-2)$, which implies $p_1|(p_3-2)<p_1$, a contradiction.

If $c=2p_2p_3$, then $p_1|c-a=2p_2(p_3-1)$, which implies $p_1|(p_3-1)<p_1$, a contradiction.

If $c=3p_2p_3$, then $p_1|c-a=p_2(3p_3-2)$. If $p_3\geq11$, we have $\frac{b}{a}=\frac{c+e-a}{a}>10$, a contradiction. If $p_3=5$, then $p_1=13$ and $p_2=19$, $b=c+e-a=3\times19\times 5+13-38=260$. Since $17<\frac{4n}{c}=\frac{52}{3}<18<19=\frac{4n}{b}$ and $\gcd(n,18)=1$, $18a=684<n$, let $m=18$ and $k=4$, then $\ind(S)=1$. If $p_3=7$, then $p_1=19$. However, we can't find a prime $p_2$ such that $7|(2p_2-19)$ and $19<p_2<38$.

{\it Subcase 2.2.} $p_2|b$.

If $b=p_2p_3$, then $p_1|b+a=p_2(p_3+2)$, which implies $p_1=p_3+2$. By Lemma 2.9, we have $p_3\leq19$, in further, $p_3\in\{5,11,17\}$.  If $p_3=5$, $p_1=7$, $p_2\leq13$, which contradicts to $n=p_1p_2p_3>1000$. If $p_3=11, p_1=13$, then $p_2\in\{17,19\}$, which contradicts to $p_3|(2p_2-p_1)$. If $p_3=17, p_1=19$, then $p_2\in\{23,29,31,37\}$, which also contradicts to $p_3|(2p_2-p_1)$.

If $b=2p_2p_3$, then $p_1|b+a=2p_2(p_3+1)$, which implies $p_1|p_3+1<p_1$, a contradiction.

If $b=3p_2p_3$, then $p_1|b+a=p_2(3p_3+2)$.  If $p_3\geq7$, we have $\frac{b}{a}\geq\frac{23}{2}>10$, a contradiction. If $p_3=5$, then $p_1=17$ and $p_2=31$, $c=b+a-e=510$. Since $\frac{6n}{c}=31<32<34=\frac{6n}{b}$ and $\gcd(n,32)=1$, $32a=1984<2635=n$, let $m=32$ and $k=6$, then $\ind(S)=1$.

{\bf Case 3.} $k=3$.

{\it Subcase 3.1.} $p_2|c$.

If $c=p_2p_3$, then $p_1|c-a=p_2(p_3-3)$, which implies $p_1|(p_3-2)<p_1$, a contradiction.

If $c=3p_2p_3$, then $p_1|c-a=3p_2(p_3-1)$, which implies $p_1|(p_3-1)<p_1$, a contradiction.

If $c=2p_2p_3$, then $p_1|c-a=p_2(2p_3-3)$, which implies $p_1=2p_3-3$.
If $p_3\geq17$, we have $\frac{b}{a}=\frac{c+e-a}{a}>10$, a contradiction. If $p_3=5$, then $p_1=7$ and $p_2<10$, a contradiction. If $p_3=7$, then $p_1=11$ and $p_2=13$, thus $\frac{n}{c}<6<\frac{n}{b}$ and $\gcd(n,6)=1$. Let $m=6$ and $k=1$, then $\ind(S)=1$. If $p_3=11$, then $p_1=19$, but there exists no prime $p_2$ such that $3p_2<4p_1$ and $p_3|(3p_2-p_1)$. If $p_3=13$, then $p_1=23$, there exists no prime $p_2$ such that $3p_2<4p_1$ and $p_3|(3p_2-p_1)$.
If $p_3=17$, which implies $p_1|49$ hence $p_1=7<p_3$, a contradiction..

{\it Subcase 3.2.} $p_2|b$.

If $b=p_2p_3$, then $p_1|b+a=p_2(p_3+3)$, which implies $2p_1|p_3+3$, a contradiction.
If $b=3p_2p_3$, then $p_1|b+a=3p_2(p_3+1)$, which implies $p_1|p_3+1<p_1$, a contradiction.

If $b=2p_2p_3$, then $p_1|b+a=p_2(2p_3+3)$.  If $p_3\geq17$, we have $\frac{b}{a}\geq\frac{37}{3}>10$, a contradiction. If $p_3=5$, then $p_1=13$ and $p_2=17$, which contradicts to $p_3|(3p_2-p_1)$. If $p_3=7$, then $p_1=17$ and there exists no suitable $p_2$. If $p_3=11$, we can't find suitable $p_1$. If $p_3=13$, then $p_1=29$, we can't find suitable $p_2$.\end{proof}

\begin{lemma} If  $a<2e$ and $a|b$, then $\ind(S)=1$.\end{lemma}
\begin{proof}
Let $m=\frac{n+a}{a}$,$m_1=\frac{n+2a}{a}$, $m_2=\frac{n+3a}{a}$, $m_3=\frac{n+4a}{a}$.

If $\gcd(n,m)=1$ then
\beqs |me|_n>\frac{n}{2}, |m(n-a)|_n=n-a>\frac{n}{2},|m(n-b)|_n=n-b>\frac{n}{2},\eeqs
thus $\ind(S)=1$.

Next assume that $\gcd(n,m)>1$. Then $\gcd(n,m_1)=\gcd(n,m_2)=\gcd(n,m_3)=1$. Moreover,
\beqs |m_1e|_n>\frac{n}{2}, |m_2e|_n>\frac{n}{2},|m_3e|_n>\frac{n}{2},|m_1a|_n<\frac{n}{2},|m_2a|_n<\frac{n}{2}, |m_3a|_n<\frac{n}{2}.\eeqs

If $b<\frac{n}{4}$, we have $|m_1(n-b)|_n=n-2b>\frac{n}{2}$. If $\frac{n}{4}<b<\frac{n}{3}$, we have $|m_3(n-b)|_n=2n-4b>\frac{n}{2}$. If $\frac{n}{3}<b<\frac{n}{2}$, we have $|m_2(n-b)|_n=2n-3b>\frac{n}{2}$.
Then we can find an integer $m_i$ such that $\gcd(n,m_i)=1$ and all of $|m_ie|_n, |m_i(n-b)|_n, |m_i(n-a)|_n$ are larger than $\frac{n}{2}$, which implies that $\ind(S)=1$.
\end{proof}

{\noindent\it Renumbering the sequence:}

Now we begin to renumber the sequence such that $e<\frac{a}{4}$. For this purpose, by Lemma 2.10, 2.11 and Lemma 2.12, we can assume that $a=p_2<2e$ and $p_2|c$.

\begin{lemma}If $a=p_2<2e$ and $a|c$, then $\ind(S)=1$ or the sequence $S$ can be renumbered as
$$(e'g)\cdot(c'g)\cdot((n-b')g)\cdot((n-a')g)$$
such that $e'<a'\leq b'<c$ and $a\geq 10e'$. Moreover, $e'=p_2$ or $e'=2p_2$.\end{lemma}
\begin{proof}

Let $m=\frac{n-a}{a}$, $m_1=\frac{n-2a}{a}$, $m_2=\frac{n+3a}{2a}$, $m_3=\frac{n+5a}{2a}$.

If $\gcd(n,m)=1$, then $\frac{n}{2}<|me|_n<n-10a$ and $|mc|_n=n-c>\frac{n}{2}$. For this case, if $|m(n-b)|_n>\frac{n}{2}$, we have $\ind(S)=1$. Otherwise, it must hold $a<|m(n-b)|_n$. We get a renumbering:
\beq e'=a, c'=|m(n-b)|_n, \{b',a'\}=\{c,n-|me|_n\},\eeq
and it is easy to check that $a'\geq10e'$.

Next we assume that $\gcd(m,n)>1$, then $p_2|(p_1p_3-1)$ and $\gcd(n,m_1)=\gcd(n,m_2)=\gcd(n,m_3)=1$.

If $c=2ta$ for some integer $t$. Let $m'=\frac{n+a}{2a}$. Then $\gcd(n,m')=1$ and $|m'e|_n<\frac{n}{2}$, $|m'c|_n=\frac{c}{2}<\frac{n}{2}$, $|m'(n-a)|_n=\frac{n-a}{2}<\frac{n}{2}$, and $\ind(S)=1$.

If $c=(2t+1)a$ for some integer $t$. We distinguish three
cases according to the value of $\frac{n}{c}$.

{\bf Case 1.} $\frac{n}{4}>c$.  Replace $m$ by $m_1$ and repeat the above process, we have $|m_1(n-b)|_n>\frac{n}{2}$,
$|m_1c|_n>\frac{n}{2}$ and $|m_1e|_n>\frac{n}{2}$, which implies $\ind(S)=1$, or we can obtain a renumbering:
\beq e'=2a, c'=|m_1(n-b)|_n, \{b',a'\}=\{2c,n-|m_1e|_n\},\eeq
it also holds that $a'\geq 10e'$.

{\bf Case 2.} $\frac{n}{4}<c<\frac{n}{3}$. Then $|m_3(n-a)|_n=\frac{n-5a}{2}<\frac{n}{2}$. We have $|m_3e|_n<\frac{n}{2}$ and $|m_3c|_n=|\frac{n+5c}{2}|_n<\frac{n}{2}$, exactly it belongs to $(\frac{n}{8},\frac{n}{3})$. Then $\ind(S)=1$.

{\bf Case 3.}  $\frac{n}{3}<c$. Then $|m_2(n-a)|_n=\frac{n-3a}{2}<\frac{n}{2}$. We have $|m_2c|_n=|\frac{n+3c}{2}|_n<\frac{n}{4}$, $|m_2e|_n<\frac{n}{2}$, and hence $\ind(S)=1$.
\end{proof}

Through the process of renumbering, we can always assume that $e\in\{p_1, p_2, 2p_2\}$ and $a>4e$. Particularly, $a\geq 10e$ when $e\in\{p_2, 2p_2\}$. Hence we also assume that $s\leq 9$ by Lemma 2.9.

Let $k_1$ be the largest positive integer such that $\lceil\frac{(k_1-1)n}{c}\rceil=\lceil\frac{(k_1-1)n}{b}\rceil$ and $\frac{k_1n}{c}\leq m<\frac{k_1n}{b}$ for some integer $m$. Since $\frac{bn}{c}\leq n-1<n=\frac{bn}{b}$ and $\frac{tn}{b}-\frac{tn}{c}=\frac{t(c-b)n}{bc}>2$ for all $t\geq b$, such integer $k_1$ always exists and $k_1\leq b$.

As mentioned above, we only need prove Proposition 2.1. We now show that Proposition 2.1 holds through the following 3 propositions.

\begin{proposition} Suppose $\lceil\frac{n}{c}\rceil<\lceil\frac{n}{b}\rceil$, then Proposition 2.1 holds. \end{proposition}
\begin{proposition} Suppose $\lceil\frac{n}{c}\rceil=\lceil\frac{n}{b}\rceil$. Let $k_1$ be the largest positive integer such that $\lceil\frac{(k_1-1)n}{c}\rceil=\lceil\frac{(k_1-1)n}{b}\rceil$ and $\frac{k_1n}{c}\leq  m_1 < \frac{k_1n}{b}$ holds for some integer $m_1$. If $k_1>\frac{b}{a}$, then Proposition 2.1 holds. \end{proposition}
\begin{proposition} Suppose $\lceil\frac{n}{c}\rceil=\lceil\frac{n}{b}\rceil$. Let $k_1$ be the largest positive integer such that $\lceil\frac{(k_1-1)n}{c}\rceil=\lceil\frac{(k_1-1)n}{b}\rceil$ and $\frac{k_1n}{c}\leq  m_1 < \frac{k_1n}{b}$ holds for some integer $m_1$. If $k_1\leq\frac{b}{a}$, then Proposition 2.1 holds. \end{proposition}

\section{Proof of Proposition 2.14}

In this section, we assume that $\lceil\frac{n}{c}\rceil<\lceil\frac{n}{b}\rceil$. Let $m_1=\lceil\frac{n}{c}\rceil$. Then we have $m_1-1 <\frac{n}c\leq m_1 < \frac{n}{b}$. By Lemma 2.3 (1), it suffices to find $m$ and $k$ such that $\frac{kn}{c}\leq m <\frac{kn}{b}$, $\gcd(m, n) =1$, $1\leq k\leq b$, and $ma < n$. So in what follows, we may always assume that $\gcd(n, m_1)> 1$.

\begin{lemma} If $\left[\frac{n}{c}, \frac{n}{b}\right]$ contains at least two integers, then ind(S) = 1.\end{lemma}
\begin{proof}
Note that $b\geq p_1p_3$ and $b\geq p_1p_3$. Thus $\frac{n}{b}\leq p_2$ and $\frac{n}{c}\leq p_1$. It must hold
\beq m_1-1<\frac{n}{c}\leq m_1<m_1+1<\frac{n}{b}=m_1+2=p_2,\eeq
or
\beq m_1-1<\frac{n}{c}\leq m_1<m_1+1\leq\frac{n}{b}<m_1+2\leq p_2.\eeq
Clearly, $p_1=m_1$ or $p_1=m_1+1$.

If (3.1) holds, we have $m_1=p_1, 2p_3|(m_1+1), m_1+2=p_2$ and $b=p_1p_3$. Then we infer that $c=p_2p_3$ and $a-e=2p_3<2p_1\leq2e$, which contradicts to the assumption $a>4e$.

If (3.2) holds, we have $m_1\geq 10$ and
\beq 2m_1-2<\frac{2n}{c}\leq 2m_1<2m_1+1<2m_1+2\leq\frac{n}{b}<2m_1+4.\eeq
If $\gcd(n,2m_1+1)=1$, let $m=2m_1+1$ and $k=2$, we have
\beqs ma<\frac{4m}{3}(a-e)=\frac{4m}{3}(c-b)<\frac{8m_1+4}{3}\times\frac{3n}{(m_1+2)(m_1-1)}\leq\frac{7n}{9}<n,\eeqs
as desired.
If $\gcd(n,2m_1+1)>1$, we infer that $p_2=2m_1+1$ and $m_1\geq28$. Let $m=3m_1+2$ and $k=3$, we have $\gcd(n,m)=1$ and
\beqs ma<\frac{4m}{3}(a-e)=\frac{4m}{3}(c-b)<\frac{12m_1+8}{3}\times\frac{3n}{(m_1+2)(m_1-1)}\leq\frac{172n}{405}<n,\eeqs
and $\ind(S)=1$.
\end{proof}

By Lemma 3.1, we may assume that $\left[\frac{n}{c}, \frac{n}{b}\right]$ contains exactly one integer $m_1$, and thus
\beq m_1-1 <\frac{n}{c}\leq m_1 <\frac{n}{b}<m_1+1.\eeq

Let $l$ be the smallest integer such that $[\frac{ln}{c}, \frac{ln}{b})$ contains at least four integers. Clearly, $l\geq 3$. Since $\frac{n}{b}-m_1 < 1$ and $m_1-\frac{n}{c} < 1$, by using the minimality of $l$ we obtain that $lm_1-4 < \frac{ln}{c} <
\frac{ln}{b} < lm_1 + 4$. Then $\frac{ln(c-b)}{bc} = \frac{ln}{b}-\frac{ln}{c} < (lm_1 + 4)-(lm_1-4) = 8$ and thus
\beqs l <\frac{8bc}{(c-b)n}<\frac{8b}{(a-e)(m_1-1)}<\frac{8b}{3e(m_1-1)}< b.\eeqs

We claim that $[\frac{ln}{c}, \frac{ln}{b})$ contains at most six integers. For any positive integer $j$, let $N_j$ denote the number of integers contained in $[\frac{ln}{c}, \frac{ln}{b})$. Since
\beqs \left(\frac{(j + 1)n}{b}-(j + 1)m_1\right)- \left(\frac{jn}{b}-jm_1\right)=\frac{n}{b}-m_1<1,\\
\left((j + 1)m_1-\frac{(j + 1)n}{c}\right)- \left(jm_1-\frac{jn}{c}\right)=m_1-\frac{n}{c}<1,\eeqs
we infer that $N_{j+1}-N_j\leq2$, it is sufficient to show our claim.

By the claim above we have
\beqs lm_1-j_0 <\frac{ln}{c}\leq lm_1-j_0 + 1 <\cdots < lm_1 -j_0 + 4 <\frac{ln}{b}\leq lm_1-j_0 + 6\eeqs
for some $1\leq j_0\leq 4$. We remark that since $n = p_1p_2p_3$ and $[\frac{ln}{c}, \frac{ln}{b})$ contains at least four integers, one of them (say $m$) must be co-prime to $n$. If $ma < n$, then we have done by Lemma 2.2(1)(with
$k = l < b$).

Proposition 2.14 can be proved by the following three lemmas.

\begin{lemma}
If $m_1\not=5,7$, then $\ind(S)=1$.
\end{lemma}
\begin{proof}
Since $m_1\not=5,7$, we have $m_1\geq 10$ and $n>m_1b\geq10b$. Let $k=l$ and let $m$ be one of the integers
in $[\frac{ln}{c}, \frac{ln}{b})$ which is co-prime to $n$.
Note that $m\leq lm_1 + 3$ and $l\geq 3$, then
\beqs ma\leq (lm_1 + 3)a <\frac{4(lm_1+3)}{3}\left(\frac{ln}{lm-j_0}-\frac{ln}{lm-j_0+6}\right)=\frac{4(lm_1+3)\times6ln}{3(lm_1-j_0)(lm_1-j_0+6)}<n.\eeqs
and we have done.
\end{proof}

\begin{lemma}
If $m_1=5$, then $\ind(S)=1$.
\end{lemma}
\begin{proof}
Since $m_1=5$, we have $4<\frac{n}{c}\leq 5<\frac{n}{b}<6$, thus $a<\frac{4}{3}(c-b)<\frac{n}{9}$.
If  $\frac{2n}{c}<9<\frac{2n}{b}$, let $m=9$ and $k=2$, then we have done.
Then $9<\frac{2n}{c}$, so $a<\frac{4}{3}(c-b)<\frac{2n}{27}<\frac{n}{13}$.

If $\frac{2n}{c}<11<\frac{2n}{b}$, we infer that $\frac{27}{2}<\frac{3n}{c}\leq15<16<\frac{33}{2}<\frac{3n}{b}<18.$
If $16a<n$, let $m=16$ and $k=3$, we have done. If $16a>a$, then $n<18a<2n$, $3n<18b<18c<4n$ and $18e<\frac{9a}{2}<n$, let $M=18$. Then $|Me|_n+|Mc|_n+|M(n-b)|_n+|M(n-a)|_n=Me+(Mc-3n)+(4n-Mb)+(2n-Ma)=3n$, then we have done.

Next we assume that $9<\frac{2n}{c}\leq10<\frac{2n}{b}\leq 11$.

We infer that $a<\frac{n}{18}$. If $\frac{3n}{c}<16<\frac{3n}{b}$, let $m=16$ and $k=3$, then we have done.
Otherwise, $\frac{27}{2}<\frac{3n}{c}\leq15<\frac{3n}{b}<16$, hence we obtain $a<\frac{n}{21}$.

If $\frac{27}{2}<\frac{3n}{c}\leq 14<15<\frac{3n}{b}<16$, we have $\frac{3n}{16}<b<\frac{n}{5}<\frac{3n}{14}<c<\frac{2n}{9}$. If $24a<n$, let $M=12$, we have
$|M(n-a)|_n>\frac{n}{2}$, $|M(n-b)|_n>\frac{n}{2}$, $|Mc|_n>\frac{n}{2}$, then $\ind(S)=1$. If $24a>n$,  we infer that $n<27a<2n$, $27e<n$ and $5n<\frac{81n}{16}<27b<27c<6n$, let $M=27$. Then
$|Me|_n+|Mc|_n+|M(n-b)|_n+|M(n-a)|_n=Me+(Mc-5n)+(6n-Mb)+(2n-Ma)=3n$, and we have done.

Next assume that $14<\frac{3n}{c}\leq15<\frac{3n}{b}<16$, we infer that $28a<n$.

If $\frac{4n}{c}\leq 19<20<21<\frac{4n}{b}$, since either $19$ or $21$ is co-prime to $n$(otherwise, $n=5\times7\times19=665<1000$, a contradiction), let $m$ be one of $19,21$ such that $\gcd(n,m)=1$ and $k=4$, then we have done.

If $\frac{4n}{c}\leq 19<20<\frac{4n}{b}\leq21$, let $M=24$, then $4n<24b<\frac{24n}{5}<5n<\frac{96n}{19}<24\times\frac{4n}{19}<24c<6n$, and $|Me|_n+|Mc|_n+|M(n-b)|_n+|M(n-a)|_n=Me+(Mc-5n)+(5n-Mb)+(n-Ma)=n$, and we have done.

If $19<\frac{4n}{c}\leq 20<21<\frac{4n}{b}$, if $\gcd(n,21)=1$, let $m=21$ and $k=4$, then we have done. If $\gcd(n,21)>1$, then $\gcd(26,n)=1$. Otherwise, $n=5\times7\times13<1000$, a contradiction. Let $M=26$, we have $4n<26b<\frac{104n}{21}<5n<\frac{26n}{5}<26c<6n$, and $|Me|_n+|Mc|_n+|M(n-b)|_n+|M(n-a)|_n=Me+(Mc-5n)+(5n-Mb)+(n-Ma)=n$. Then $\ind(S)=1$.

Next assume that $19<\frac{4n}{c}\leq 20<\frac{4n}{b}\leq 21$.

If $39<\frac{8n}{c}\leq 40<\frac{8n}{b}\leq 41$, we infer that $a<\frac{n}{73}$. Then $b>\frac{n}{6}>\frac{73a}{6}>12a$, which contradicts to $s\leq 9$.

If $39<\frac{8n}{c}\leq 40<41<\frac{8n}{b}\leq 42$, we infer that $a<\frac{n}{51}$. Let $M=36$, then $|Me|_n<9a<\frac{n}{2}$, $|Mc|_n<\frac{n}{2}$ and $|M(n-a)|_n<\frac{n}{2}$(otherwise, $72a<n$ contradicts to $s\leq9$).  Exactly,  $|Mc|_n$ belongs to  $(\frac{n}{5},\frac{15n}{39})$. Then $\ind(S)=1$.

If $38<\frac{8n}{c}\leq 39<40<\frac{8n}{b}\leq 41$, we infer that $\frac{n}{52}<a<\frac{n}{48}$. If $\gcd(n,39)=1$, let $m=39$ and $k=8$, then we have done. If $\gcd(n,39)>1$, then $\gcd(n,11)=1$, otherwise $n=5\times 11\times 13=715<1000$, a contradiction. Let $M=44$, then $\gcd(n,M)=1$ and $|M(n-a)|_n<\frac{n}{2}$(otherwise, $88a<n$ contradicts to $s\leq9$), $|M(n-b)|_n<\frac{n}{2}$, $|Mc|_n<\frac{n}{2}$. Exactly,  $|M(n-b)|_n$ belongs to  $(\frac{n}{5},\frac{13n}{41})$ and $|Mc|_n$ belongs to $(\frac{n}{39},\frac{3n}{19})$. Then $\ind(S)=1$.

Now we infer that $38<\frac{8n}{c}\leq39<40<41<\frac{8n}{b}\leq 42$ and $\frac{n}{53}<a<\frac{n}{37}$.
Let $M=36$, we have $|M(n-a)|_n<\frac{n}{2}$(otherwise, $72a<n$ contradicts to $s\leq9$, $Me<9a<\frac{n}{2}$ and $6n+\frac{6n}{7}\leq Mb<7n+\frac{n}{41}$.
If $Mb<7n$, then $|M(n-b)|_n<\frac{n}{2}$ and $\ind(S)=1$. Hence we infer that $a<\frac{n}{46}$. Moreover, we can assume that $n=5\times13\times41$. Otherwise, there exists an integer(say $m$) between $39$ and $41$ such that $\gcd(n,m)=1$ and $ma<n$. Let $k=8$, thus $\ind(S)=1$. Simply calculating shows that $p_1=5, p_2=13$. However, we can't find suitable $a$ and $e$ such that $a>4e$.

Hence we complete the proof.
\end{proof}

\begin{lemma}
If $m_1=7$, then $\ind(S)=1$.
\end{lemma}
\begin{proof}
Since $m_1=7$, we have $6<\frac{n}{c}\leq 7<\frac{n}{b}<8$, thus $a<\frac{4}{3}(c-b)<\frac{n}{18}$.

If $12<\frac{2n}{c}\leq 13<14<15<\frac{n}{b}<16$, then at least one of $13,14,15$ is co-prime to $n$. Let $m\in[13,15]$ such that $\gcd(n,m)=1$ and $k=1$, then $ma<n$ and $\ind(S)=1$.

If $13<\frac{2n}{c}\leq 14<15<\frac{2n}{b}<16$,  we have $\frac{n}{8}<b<\frac{2n}{15}<\frac{n}{7}<c<\frac{2n}{13}$, $a<\frac{n}{26}$ and $2n<16b<2n+\frac{2n}{15}<2n+\frac{2n}{7}<16c<2n+\frac{6n}{13}$. If $32a>n$, let $M=16$, we have
$|Me|_n<\frac{n}{2}$, $|Mc|_n<\frac{n}{2}$ and $|M(n-a)|_n<\frac{n}{2}$, then $\ind(S)=1$. If $32a<n$, by inequality $\frac{4n}{c}\leq 28<29<30<\frac{4n}{b}$, we infer that $n=5\times 7\times 29$. Since $e<\frac{a}{4}<\frac{n}{104}$, we have $e<10$, then $e=p_1$. If $p_1=7$, then  $p_2=29$ and $c=5\times29$, $a=29$, thus $b=4\times 29+7$, which contradicts to $7|b$. We infer that $p_1=5$, and thus $\frac{n}{c}\leq\frac{n}{p_2p_3}=5$, a contradiction.

If $12<\frac{2n}{c}\leq 13<14<\frac{2n}{b}<15$, we infer that $a<\frac{n}{22}$ and $91|n$. we also assume that $27a>n$. Otherwise, let $m=27$ and $k=4$, we have $\frac{4n}{c}\leq 26<27<28<\frac{42n}{b}$.
If $5|n$, then $n=5\times91=455<1000$, a contradiction. Thus $\gcd(n,30)=1$. Let $M=30$, we have
$Me<8a<n$, $4n<30b<30c<5n$ and $n<30a<2n$. Then $|Me|_n+|Mc|_n+|M(n-b)|_n+|M(n-a)|_n=Me+(Mc-4n)+(5n-Mb)+(2n-Ma)=3n$, and $\ind(S)=1$.

Next assume that $13<\frac{2n}{c}\leq 14<\frac{2n}{b}<15$, and we infer that $a<\frac{n}{36}$.

If $\frac{4n}{c}<27<\frac{4n}{b}$, let $m=27$  and $k=4$, then we have done. So $27<\frac{4n}{c}\leq 28<\frac{4n}{b}<30$, and $a<\frac{n}{50}$.

If $27<\frac{4n}{c}\leq 28<29<\frac{4n}{b}<30$, we have $\frac{2n}{15}<b<\frac{4n}{29}<\frac{n}{7}\leq c<\frac{4n}{27}$ and $\frac{24n}{5}<b<\frac{144n}{29}<5n<\frac{36n}{7}\leq c<\frac{48n}{9}<6n$. Let $M=36$, we have $|Me|_n+|Mc|_n+|M(n-b)|_n+|M(n-a)|_n=Me+(Mc-5n)+(5n-Mb)+(n-Ma)=n$, and $\ind(S)=1$.

If $27<\frac{4n}{c}\leq 28<\frac{4n}{b}\leq29$, we infer that $a<\frac{n}{73}$, then $b>\frac{4n}{29}>\frac{292a}{29}>10a$, which contradicts to $s\leq 9$.

We complete the proof.
\end{proof}

\section{Proof of Proposition 2.15}

In this section, we always assume that $\lceil\frac{n}{c}\rceil=\lceil\frac{n}{b}\rceil$, so $k_1\geq2$, and we also assume that $k_1>\frac{b}{a}$. Proposition 2.15 can be proved through the following two lemmas.

\begin{lemma}If the assumption is as in Proposition 2.15, then $k_1<3$.\end{lemma}
\begin{proof}If $k_1\geq3$, then $\frac{(k_1-1)n}{b}-\frac{(k_1-1)n}{c}=\frac{(a-e)(k_1-1)n}{bc}\geq\frac{3a}{4}\frac{2k_1n}{3bc}>1$, a contradiction. \end{proof}

\begin{lemma}If the assumption is as in Proposition 2.15 and $k_1=2$, then $\ind(S)=1$.\end{lemma}
\begin{proof}

If $\frac{n}{c}>3$, then $\frac{n}{b}-\frac{n}{c}=\frac{(a-e)n}{bc}\geq\frac{2a}{3}\frac{n}{bc}>1$, a contradiction.

If $\frac{n}{c}\leq3<\frac{n}{b}$, we have $n<3c<2n$, $3a<3b<n$. Let $m=3$, then $\gcd(n,m)=1$ and
$|me|_n+|mc|_n+|m(n-b)|_n+|m(n-a)|_n=me+(mc-n)+(n-mb)+(n-ma)=n$, we have done.

If $\frac{n}{c}<\frac{n}{b}<3$, then $\frac{n}{3}<b<2a$, and $2n<6c<3n$, $2n<6b<3n$, $6a>3b>n$. $6e<2a<n$. Let $m=6$, then $\gcd(n,m)=1$, and $3n\geq|me|_n+|mc|_n+|m(n-b)|_n+|m(n-a)|_n\geq me+(mc-2n)+(3n-mb)+(2n-ma)=3n$, we have done.
\end{proof}

\section{Proof of Proposition 2.16}

In this section, we always assume that $\lceil\frac{n}{c}\rceil=\lceil\frac{n}{b}\rceil$, so $k_1\geq2$, and we also assume that $k_1<\frac{b}{a}$, hence $s\geq k_1$. Proposition 2.16 can be proved by the following Lemmas 5.1-5.6 and Lemma 5.9.

\begin{lemma}If the assumption is as in Proposition 2.16, then $k_1\leq7$.\end{lemma}
\begin{proof}If $k_1\geq8$, then  $\frac{(k_1-1)n}{b}-\frac{(k_1-1)n}{c}\geq\frac{(a-e)7n}{bc}\geq\frac{21an}{4bc}>\frac{21}{20}>1$, a contradiction.\end{proof}

\begin{lemma}If the assumption is as in Proposition 2.16 and $k_1=7$, then $\ind(S)=1$.\end{lemma}
\begin{proof}
If $c\leq\frac{9n}{20}$, we have  $\frac{(k_1-1)n}{b}-\frac{(k_1-1)n}{c}\geq\frac{(a-e)6n}{bc}\geq\frac{9an}{2bc}\geq\frac{10a}{b}>1$, a contradiction.
Thus $\frac{n}{c}<\frac{20}{9}$ and we infer that $12<\frac{6n}{c}<\frac{6n}{b}\leq13$ or $13<\frac{6n}{c}<\frac{6n}{b}\leq14$.

{\bf Case 1.} It holds that $12<\frac{6n}{c}<\frac{6n}{b}\leq13$.  Then $a<\frac{n}{19}$ and we have \beqs &14<\frac{7n}{c}\leq 15<\frac{7n}{b}\leq\frac{91}{6},\\
&16<\frac{8n}{c}\leq 17<\frac{8n}{b}\leq\frac{52}{3},\\
&18<\frac{9n}{c}\leq 19<\frac{9n}{b}\leq\frac{39}{2},\\
&20<\frac{10n}{c}\leq 21<\frac{10n}{b}\leq\frac{65}{3}.\eeqs

If $\gcd(n,15)=1$, let $m=15$ and $k=7$, then we have done. If $\gcd(n,17)=1$, let $m=17$ and $k=8$, then we have done. If $\gcd(n,19)=1$, let $m=19$ and $k=9$, then we have done.

Now assume that $n=5\times 17\times19$ and thus $\gcd(n,21)=1$. If $21a<n$, let $m=21$ and $k=10$, then we have done.
If $21a>n$, let $M=12$, we have $|M(n-a)|_n<\frac{n}{2}$ and $|Me|_n<\frac{n}{2}$. Moreover, $5n+\frac{7n}{13}=\frac{72n}{13}<12b<\frac{28n}{5}=5n+\frac{3n}{5}$, which implies that $|M(n-b)|_n<\frac{n}{2}$. So $\ind(S)=1$.

{\bf Case 2.} It holds that $13<\frac{6n}{c}<\frac{6n}{b}\leq14$.  Then $a<\frac{n}{22}$  and $\frac{91}{6}<\frac{7n}{c}< 16<\frac{7n}{b}\leq\frac{49}{3}$. Let $m=16$ and $k=7$, we have $\gcd(n,m)=1$ and $ma<n$, then $\ind(S)=1$.
\end{proof}

\begin{lemma}If the assumption is as in Proposition 2.16 and $k_1=6$, then $\ind(S)=1$.\end{lemma}
\begin{proof}
Similar to Lemma 5.2, we have $\frac{n}{c}<\frac{8}{3}$ and we infer that $10<\frac{5n}{c}<\frac{5n}{b}\leq11$, or $11<\frac{5n}{c}<\frac{5n}{b}\leq12$, or $12<\frac{5n}{c}<\frac{5n}{b}\leq13$, or $13<\frac{5n}{c}<\frac{5n}{b}\leq\frac{40}{3}$.

{\bf Case 1.} It holds that $10<\frac{5n}{c}<\frac{5n}{b}\leq11$, then $a<\frac{n}{16}$.  If $18a>n$.
Let $M=9$, we infer that $|M(n-a)|_n<\frac{n}{2}$, $|Me|_n<\frac{n}{2}$ and $|Mc|_n<\frac{n}{2}$(exactly, $4n+\frac{n}{11}<9c<4n+\frac{n}{2}$). Then $\ind(S)=1$. Moreover, we have
\beqs \frac{6n}{c}\leq13<\frac{6n}{b},&\frac{7n}{c}\leq15<\frac{7n}{b},& \frac{8n}{c}\leq17<\frac{8n}{b}.\eeqs
If $\gcd(n,13)=1$, let $m=13$ and $k=6$, if $\gcd(n,15)=1$, let $m=15$ and $k=7$, if $\gcd(n,17)=1$, let $m=17$ and $k=8$, then $ma<n$ and $\ind(S)=1$. If none of the three integers is co-prime to $n$, then $n=5\times13\times17$ and $p_1=13, p_2=17, p_3=5$.

By the renumbering process, we may assume that $17\leq e\leq\frac{a}{10}$ or $a\geq 4\times17$. If $17\leq e\leq\frac{a}{10}$, then $e\leq \frac{n}{180}<10$, a contradiction. If $a\geq 4\times17$, then $a\geq\frac{4n}{5\times13}>\frac{n}{17}$, a contradiction.

{\bf Case 2.} It holds that $11<\frac{5n}{c}<\frac{5n}{b}\leq12$. Then $a<\frac{4}{3}(a-e)=\frac{4}{3}(c-b)<\frac{4}{3}(\frac{5n}{11}-\frac{n}{12})<\frac{n}{19}$.
If $\frac{8n}{c}<18<\frac{8n}{b}$, let $m=18$ and $k=8$, then we have done. Otherwise, it holds  $18<\frac{8n}{c}\leq19<\frac{8n}{b}\leq\frac{96}{5}$, and thus $a<\frac{4}{3}(\frac{4n}{9}-\frac{5}{12})=\frac{n}{27}$. We infer that $\frac{b}{a}>\frac{5n}{12}\times\frac{27}{n}>10$, which contradicts to $s\leq9$.

{\bf Case 3.} It holds that $12<\frac{5n}{c}<\frac{5n}{b}\leq13$. Then $a<\frac{4}{3}(a-e)=\frac{4}{3}(c-b)<\frac{4}{3}(\frac{5n}{12}-\frac{n}{13})<\frac{n}{23}$.
If $\frac{7n}{c}<18<\frac{7n}{b}$, let $m=18$ and $k=7$, then we have done. Otherwise, it holds  $\frac{84}{5}<\frac{7n}{c}\leq17<\frac{7n}{b}\leq18$, and thus $a<\frac{4}{3}(\frac{5n}{12}-\frac{7}{18})=\frac{n}{27}$. We infer that $\frac{b}{a}>\frac{7n}{18}\times\frac{27}{n}>10$, which contradicts to $s\leq9$.

{\bf Case 4.} It holds that $13<\frac{5n}{c}<\frac{5n}{b}\leq\frac{40}{3}$.
Then $a<\frac{4}{3}(a-e)=\frac{4}{3}(c-b)<\frac{4}{3}(\frac{5n}{13}-\frac{3n}{8})=\frac{n}{78}$. Since $s\geq k_1=6$, by Lemma 2.7, we have $b>\frac{n}{4}$, thus $\frac{b}{a}>\frac{78}{4}>19$, which contradicts to $s\leq9$.
\end{proof}

\begin{lemma}If the assumption is as in Proposition 2.16 and $k_1=5$, then $\ind(S)=1$.\end{lemma}
\begin{proof}
Similar to Lemma 5.2, we have $\frac{n}{c}<\frac{10}{3}$, then $8+t<\frac{4n}{c}<\frac{4n}{b}\leq9+t$ for some $t\in[0,4]$ or $13<\frac{4n}{c}<\frac{4n}{b}\leq\frac{40}{3}$. We distinguish six cases.

{\bf Case 1.}  It holds that $13<\frac{4n}{c}<\frac{4n}{b}\leq\frac{40}{3}$.
We have $39<\frac{12n}{c}<\frac{12n}{b}\leq40$, which contradicts to the maximality of $k_1$.

{\bf Case 2.}  $t=0$.

It holds that $8<\frac{4n}{c}<\frac{4n}{b}\leq9$, and we infer that $a<\frac{n}{13}$. Moreover, we have
\beqs \frac{5n}{c}\leq11<\frac{5n}{b},&\frac{6n}{c}\leq13<\frac{6n}{b},& \frac{7n}{c}\leq15<\frac{7n}{b}<16.\eeqs
If $16a>n$, let $M=16$, then $Me<n$, $n<Ma<2n$ and $7n<Mb<Mc<8n$. We infer that $|Me|_n+|Mc|_n+|M(n-b)|_n+|M(n-a)|_n=3n$, and $\ind(S)=1$. Thus $16a<n$. If $\gcd(n,11)=1$, let $m=11$ and $k=5$, if $\gcd(n,13)=1$, let $m=13$ and $k=6$, if $\gcd(n,15)=1$, let $m=15$ and $k=7$, then $ma<n$ and $\ind(S)=1$.
If none of $11, 13, 15$ is co-prime to $n$, then $n=5\times11\times13=715<1000$, a contradiction.

{\bf Case 3.}  $t=1$.

It holds that $9<\frac{4n}{c}<\frac{4n}{b}\leq10$. We infer that $a<\frac{n}{16}$ and $\frac{45}{4}<\frac{5n}{c}<12<\frac{5n}{b}\leq\frac{25}{2}$. Let $m=12$ and $k=5$, then we have done.

{\bf Case 4.}  $t=2$.

It holds that $10<\frac{4n}{c}<\frac{4n}{b}\leq11$. We infer that $a<\frac{n}{20}$ and $15<\frac{6n}{c}<16<\frac{6n}{b}\leq\frac{33}{2}$. Let $m=16$ and $k=6$, then we have done.

{\bf Case 5.}  $t=3$.

It holds that $11<\frac{4n}{c}<\frac{4n}{b}\leq12$, and we infer that $a<\frac{n}{24}$. Moreover, we have
\beqs \frac{5n}{c}\leq14<\frac{5n}{b}\leq15,&\frac{6n}{c}\leq17<\frac{6n}{b}\leq18,& \frac{7n}{c}\leq20<\frac{7n}{b}\leq21.\eeqs
If $\gcd(n,14)=1$, let $m=14$ and $k=5$, if $\gcd(n,17)=1$, let $m=17$ and $k=6$, if $\gcd(n,20)=1$, let $m=20$ and $k=7$, then $ma<n$ and $\ind(S)=1$. If none of the three integers is co-prime to $n$, then $n=5\times7\times17=595<1000$, a contradiction.

{\bf Case 6.}  $t=4$.

It holds that $12<\frac{4n}{c}<\frac{4n}{b}\leq13$, then $a<\frac{4}{3}(\frac{n}{12}-\frac{n}{13})<\frac{n}{29}$ and
$15<\frac{5n}{c}<16<\frac{5n}{b}\leq\frac{65}{4}$. Let $m=16$ and $k=5$, we have $\ind(S)=1$.
\end{proof}

\begin{lemma}If the assumption is as in Proposition 2.16 and $k_1=4$, then $\ind(S)=1$.\end{lemma}
\begin{proof}
Similar to Lemma 5.2, we have $\frac{n}{c}<\frac{40}{9}$,  then $6+t<\frac{3n}{c}<\frac{3n}{b}\leq7+t$ for some $t\in[0,6]$ or $13<\frac{3n}{c}<\frac{3n}{b}\leq\frac{40}{3}$. We distinguish eight cases.

{\bf Case 1.}  It holds that $13<\frac{3n}{c}<\frac{3n}{b}\leq\frac{40}{3}$.
We have $39<\frac{9n}{c}<\frac{9n}{b}\leq40$, which contradicts to the maximality of $k_1$.

{\bf Case 2.}  $t=0$.

It holds that $6<\frac{3n}{c}<\frac{3n}{b}\leq7$. We infer that $a<\frac{n}{10}$ and $8<\frac{4n}{c}<9<\frac{4n}{b}\leq\frac{28}{3}$. Let $m=9$ and $k=4$, we have $\ind(S)=1$.

{\bf Case 3.}  $t=1$.

It holds that $7<\frac{3n}{c}<\frac{3n}{b}\leq8$, and we infer that $a<\frac{n}{13}$.
If $\frac{5n}{c}<12<\frac{5n}{b}$, let $m=12$ and $k=5$, then $\ind(S)=1$. We may assume that $12<\frac{5n}{c}\leq13<\frac{5n}{b}<\frac{40}{3}$, then $a<\frac{n}{18}$. Let $M=18$, we have
$$8n>\frac{15n}{2}>18c>\frac{36n}{5}>7n>\frac{90n}{13}>18b>\frac{27n}{4}>6n,$$
and $|Me|_n+|Mc|_n+|M(n-b)|_n+|M(n-a)|_n=Me+(Mc-7n)+(7n-Mb)+(n-Ma)=n$, then $\ind(S)=1$.

{\bf Case 4.}  $t=2$.

It holds that $8<\frac{3n}{c}<\frac{3n}{b}\leq9$, and we infer that $a<\frac{n}{18}$. Since
\beqs \frac{4n}{c}\leq 11<\frac{4n}{b}, \frac{5n}{c}\leq 14<\frac{5n}{b}, \frac{6n}{c}\leq 17<\frac{6n}{b},\eeqs
we infer that $n=7\times11\times17$.

If $18<\frac{56}{3}<\frac{7n}{c}<19<\frac{7n}{b}\leq20$, we have $a<\frac{n}{30}$, let $m=19$ and $k=7$. Then $\gcd(n,m)=1$ and $\ind(S)=1$.

If $19<\frac{7n}{c}<20<\frac{7n}{b}\leq21$, we have $a<\frac{n}{21}$, let $m=20$ and $k=7$. Then $\gcd(n,m)=1$ and $\ind(S)=1$.

If $\frac{56}{3}<\frac{7n}{c}<19<20<\frac{7n}{b}\leq21$, we infer that $19a>n$. If $27c<10n$, then
$a<\frac{4}{3}(c-b)<\frac{4}{3}(\frac{10n}{27}-\frac{n}{3})<\frac{n}{20}$, a contradiction. So
$\frac{3n}{8}>c>\frac{10n}{27}>\frac{7n}{20}>b>\frac{n}{3}$ and $\frac{81n}{8}>27c>10n>\frac{189n}{20}>27b>9n$, $n<27a<\frac{3n}{2}$. Let $M=27$, we have $|Mc|_n>\frac{n}{2}$, $|M(n-b)|_n>\frac{n}{2}$, $|M(n-a)|_n>\frac{n}{2}$, and $\ind(S)=1$.

{\bf Case 5.}  $t=3$.

It holds that $9<\frac{3n}{c}<\frac{3n}{b}\leq10$, and we infer that $a<\frac{n}{22}$ and $\frac{5n}{c}<16<\frac{5n}{b}$. Let $m=16$ and $k=5$, then $\ind(S)=1$.

{\bf Case 6.}  $t=4$.

It holds that $10<\frac{3n}{c}<\frac{3n}{b}\leq11$, and we infer that $a<\frac{n}{27}$.
If $\frac{5n}{c}<18<\frac{5n}{b}$, let $m=18$ and $k=5$, then we have done. If $\frac{7n}{c}<24<\frac{7n}{b}$, let $m=24$ and $k=7$, then we have done. Otherwise, we have $\frac{5n}{18}<b<c<\frac{7n}{24}$ and $a<\frac{4}{3}(c-b)<\frac{n}{54}$. Then $b>\frac{5n}{18}>\frac{5\times54a}{18}=15a$, which contradicts to $s\leq 9$.

{\bf Case 7.}  $t=5$.

It holds that $11<\frac{3n}{c}<\frac{3n}{b}\leq12$, and we infer that $a<\frac{n}{33}$ and \beqs \frac{4n}{c}\leq15<\frac{5n}{b}, \frac{5n}{c}\leq19<\frac{5n}{b}, \frac{6n}{c}\leq23<\frac{6n}{b}.\eeqs
If $\gcd(n,17)=1$, let $m=17$ and $k=4$, if $\gcd(n,21)=1$, let $m=21$ and $k=5$, if $\gcd(n,25)=1$, let $m=25$ and $k=6$, then $ma<n$ and $\ind(S)=1$. If none of the three integers is co-prime to $n$, then there exists an integer $m\in[26, 27]$ belongs to $[\frac{7n}{c},\frac{7n}{b})$, let $k=7$, then $\gcd(m,n)=1$ and $\ind(S)=1$.

{\bf Case 8.}  $t=6$.

It holds that $12<\frac{3n}{c}<\frac{3n}{b}\leq13$, then we infer that $a<\frac{n}{39}$ and \beqs \frac{4n}{c}\leq17<\frac{5n}{b}, \frac{5n}{c}\leq21<\frac{5n}{b}, \frac{6n}{c}\leq25<\frac{6n}{b}.\eeqs
If $\gcd(n,17)=1$, let $m=17$ and $k=4$, if $\gcd(n,21)=1$, let $m=21$ and $k=5$, if $\gcd(n,25)=1$, let $m=25$ and $k=6$, then $ma<n$ and $\ind(S)=1$. If none of the three integers is co-prime to $n$, then $n=5\times7\times17=595<1000$, a contradiction.
\end{proof}

\begin{lemma}If the assumption is as in Proposition 2.16 and $k_1=3$, then $\ind(S)=1$.\end{lemma}
\begin{proof}
Similar to Lemma 5.2, we have $\frac{n}{c}<\frac{20}{3}$,  then $4+t<\frac{2n}{c}<\frac{2n}{b}\leq5+t$ for some $t\in[0,8]$ or $13<\frac{2n}{c}<\frac{2n}{b}\leq\frac{40}{3}$. We distinguish ten cases.

{\bf Case 1.}  It holds that $13<\frac{2n}{c}<\frac{2n}{b}\leq\frac{40}{3}$.
We have $39<\frac{6n}{c}<\frac{6n}{b}\leq40$, which contradicts to the maximality of $k_1$.

{\bf Case 2.}  $t=0$.

It holds that $4<\frac{2n}{c}<\frac{2n}{b}\leq5$. We infer that $7a<n$, $6<\frac{3n}{c}\leq 7<\frac{3n}{b}\leq\frac{15}{2}$, and $8<\frac{4n}{c}<9<\frac{4n}{b}\leq10$.  If $9a<n$, let $m=9$ and $k=4$, then we have done. If $9a>n$, let $M=18$, then $7n<\frac{36n}{5}<18b<\frac{54n}{7}<8n<18c<9n$, and $18e<5a<n$. Then $|Me|_n+|Mc|_n+|M(n-b)|_n+|M(n-a)|_n=3n$, and $\ind(S)=1$.

{\bf Case 3.}  $t=1$.

It holds that $5<\frac{2n}{c}<\frac{2n}{b}\leq6$.  We infer that $11a<n$ and $\frac{3n}{c}<8<\frac{3n}{b}$. Let $m=8$ and $k=3$, then $\ind(S)=1$.

{\bf Case 4.}  $t=2$.

It holds that $6<\frac{2n}{c}<\frac{2n}{b}\leq7$.  We infer that $15a<n$ and
\beqs \frac{3n}{c}<10<\frac{3n}{b}, \frac{4n}{c}<13<\frac{4n}{b},\eeqs
thus $\gcd(n,10)>1, \gcd(n,13)>1$.

{\it Subcase 4.1.} $\frac{5n}{c}<16<\frac{5n}{b}\leq\frac{35}{2}$. If $16a<n$, let $m=16$ and $k=5$, then we have done. If $16a>n$, let $M=18$, we have  $5n<18b<18c<6n$ and $|Me|_n+|Mc|_n+|M(n-b)|_n+|M(n-a)|_n=3n$. Then $\ind(S)=1$.

{\it Subcase 4.2.} $16<\frac{5n}{c}\leq17<\frac{5n}{b}\leq\frac{35}{2}$. We infer that $28a<n$. Then $\gcd(n,17)>1$ and $n=5\times13\times17$. Let $M=27$, we have $$9n>\frac{135n}{16}>27c>\frac{81n}{10}>8n=\frac{136n}{17}>\frac{135n}{17}>27b>\frac{54n}{7}>7n,$$
then $|Me|_n+|Mc|_n+|M(n-b)|_n+|M(n-a)|_n=3n$ and $\ind(S)=1$.

{\bf Case 5.}  $t=3$.

It holds that $7<\frac{2n}{c}<\frac{2n}{b}\leq8$.  We infer that $21a<n$ and
\beqs \frac{3n}{c}<11<\frac{3n}{b}, \frac{4n}{c}<15<\frac{4n}{b},\eeqs
thus $\gcd(n,11)>1, \gcd(n,15)>1$.

{\it Subcase 5.1.} $\frac{5n}{c}<18<\frac{5n}{b}\leq20$. Let $m=18$ and $k=5$, then we have done.

{\it Subcase 5.2.} $18<\frac{5n}{c}\leq19<\frac{5n}{b}\leq20$. We infer that $27a<n$ and thus $\gcd(n,19)>1$. In further, $\frac{7n}{c}\leq\frac{77}{3}<26<\frac{133}{5}<\frac{7n}{b}$, let $m=26$ and $k=7$, then $\gcd(n,26)=1$ and $\ind(S)=1$.

{\bf Case 6.}  $t=4$.

It holds that $8<\frac{2n}{c}<\frac{2n}{b}\leq9$.  We infer that $27a<n$ and
\beqs \frac{3n}{c}<13<\frac{3n}{b}, \frac{4n}{c}<17<\frac{4n}{b},\eeqs
thus $\gcd(n,13)>1, \gcd(n,17)>1$.

{\it Subcase 6.1.} $\frac{5n}{c}\leq21<22<\frac{5n}{b}$. Let $k=5$ and $m\in[21,22]$ such that $\gcd(n,m)=1$, then we have done.

{\it Subcase 6.2.} $21<\frac{5n}{c}\leq22<\frac{5n}{b}\leq\frac{45}{2}$. We infer that $a<\frac{4n}{189}$ and thus $\frac{b}{a}>\frac{2}{9}\times\frac{189}{4}=\frac{21}{2}>10$, which contradicts to $s\leq9$.

{\it Subcase 6.3.} $20<\frac{5n}{c}\leq21<\frac{5n}{b}\leq22$. We infer that $33a<n$, $7|n$ and thus $\gcd(n,33)=1$.
Let $M=33$, we have $\frac{15n}{2}<33b<\frac{99n}{13}$, $\frac{n}{2}<33a<n$(otherwise, $66a<n<\frac{22b}{5}$, thus $b>15a$, a contradiction) and $33e<9a<n$. Then $|Me|_n<\frac{n}{2}$,  $|M(n-b)|_n<\frac{n}{2}$,  $|M(n-a)|_n<\frac{n}{2}$, and $\ind(S)=1$.

{\bf Case 7.}  $t=5$.

It holds that $9<\frac{2n}{c}<\frac{2n}{b}\leq10$.  We infer that $33a<n$. If $\frac{5n}{c}<24<\frac{5n}{b}$. Let $k=5$ and $m=24$, then $\gcd(n,m)=1$ and $\ind(S)=1$. Otherwise, assume that $\frac{45}{2}<\frac{5n}{c}\leq23<\frac{5n}{b}<24$. We infer that $54a<n$ and $\frac{b}{a}>\frac{5}{24}\times54=\frac{90}{8}>10$, which contradicts to $s\leq9$.

{\bf Case 8.}  $t=6$.

It holds that $10<\frac{2n}{c}<\frac{2n}{b}\leq11$.  We infer that $41a<n$ and $\frac{3n}{c}<16<\frac{3n}{b}$. Let $k=3$ and $m=16$, then $\gcd(n,m)=1$ and $\ind(S)=1$.

{\bf Case 9.}  $t=7$.

It holds that $11<\frac{2n}{c}<\frac{2n}{b}\leq12$.  We infer that $49a<n$ and
\beqs \frac{3n}{c}<17<\frac{3n}{b}, \frac{4n}{c}<23<\frac{4n}{b},\eeqs
thus $\gcd(n,17)>1, \gcd(n,23)>1$.

{\it Subcase 9.1.} $\frac{5n}{c}\leq28<29<\frac{5n}{b}$. Let $k=5$ and $m\in[28,29]$ such that $\gcd(n,m)=1$, then we have done.

{\it Subcase 9.2.} $28<\frac{5n}{c}\leq29<\frac{5n}{b}\leq30$. We infer that $a<\frac{n}{63}$ and thus $\frac{b}{a}>\frac{21}{2}>10$, which contradicts to $s\leq9$.

{\it Subcase 9.3.} $\frac{55}{2}<\frac{5n}{c}\leq28<\frac{5n}{b}\leq29$. Similar to {\it Subcase 9.2.}, we get a contradiction.

{\bf Case 10.}  $t=8$.

It holds that $12<\frac{2n}{c}<\frac{2n}{b}\leq13$, we infer that $a<\frac{n}{58}$.
If $\frac{5n}{c}<32<\frac{5n}{b}$, let $m=32$ and  $k=5$, then we have done. Otherwise, we have
$30<\frac{5n}{c}<\frac{5n}{b}<32$ and $a<\frac{4}{3}(a-e)=\frac{4}{3}(c-b)<\frac{4}{3}(\frac{n}{6}-\frac{5n}{32})=\frac{n}{72}$. Then
$b>\frac{5n}{32}>\frac{5\times72a}{32}=\frac{45a}{4}>10a$, which contradicts to $s\leq9$.
\end{proof}

\begin{lemma}
Suppose that $p_3=5$, $p_1\geq 13$ and $6a<n$. If $e=2p_2$,  then $p_1\geq23$ and $n>57e$. If $a\leq p_1$, then  $n>85e$.
\end{lemma}
\begin{proof}
This result can be checked directly.
\end{proof}

\begin{lemma}Let the assumption be as in Proposition 2.16. If $k_1=2$, $4<\frac{2n}{c}\leq5<\frac{2n}{b}\leq6$ and $5|n$, then $\ind(S)=1$.\end{lemma}
\begin{proof}
If $n<6a$, $me<2a<n$, $2n<6b<6c<3n$, let $M=6$. Then $|Me|_n+|Mc|_n+|M(n-b)|_n+|M(n-a)|_n\geq Me+(Mc-2n)+(3n-Mb)+(2n-Ma)=3n$, and $\ind(S)=1$. Next we assume that $6a<n$ and  distinguish three cases.

{\bf Case 1.} $7<\frac{3n}{c}\leq8<\frac{3n}{b}\leq9$.

If $8a<n$, let $m=8$ and $k=3$, we have $\ind(S)=1$. If $8a>n$, let $M=9$, we have $3n<9b<9c<4n$, $9e<3a<n$ and $9a<2n$, so $|Me|_n+|Mc|_n+|M(n-b)|_n+|M(n-a)|_n=Me+(Mc-3n)+(4n-Mb)+(2n-Ma)=3n$. Then $\ind(S)=1$.

{\bf Case 2.} $6<\frac{3n}{c}\leq7<\frac{3n}{b}\leq8$.

If $8a>n$, let $M=8$, we have $3n<8b<8c<4n$, $8e<2a<n$ and $8a<2n$, so $|Me|_n+|Mc|_n+|M(n-b)|_n+|M(n-a)|_n=Me+(Mc-3n)+(4n-Mb)+(2n-Ma)=3n$. Then $\ind(S)=1$.

Next assume that $8a<n$, then $7|n$, $\gcd(n,11)=\gcd(n,13)=1$ and $\frac{11n}{2}>11c>\frac{33n}{7}>\frac{22n}{5}>11b>\frac{33n}{8}$.

If $11c<5n$, let $M=12$, we have $\frac{60n}{11}>12c>\frac{36n}{7}>\frac{24n}{5}>12b>\frac{9n}{2}$ and $12e<3a<\frac{n}{2}$. Then $|Me|_n<\frac{n}{2}$, $|Mc|_n<\frac{n}{2}$, $|M(n-b)|_n<\frac{n}{2}$, and $\ind(S)=1$.

If $11c>5n$ and $11a<n$, let $M=11$, then $|Me|_n+|Mc|_n+|M(n-b)|_n+|M(n-a)|_n=Me+(Mc-5n)+(5n-Mb)+(n-Ma)=n$, and  $\ind(S)=1$.

Let $11c>5n$ and $11a>n$. If $18b>7n$, let $M=9$, we have $\frac{9n}{2}>9c>\frac{45n}{11}>\frac{18n}{5}>9b>\frac{7n}{2}$, then $|Me|_n<\frac{n}{2}, |Mc|_n<\frac{n}{2}, |M(n-b)|_n<\frac{n}{2}$, and  $\ind(S)=1$. If $18b<7n$, let $M=18$, we have  $|Me|_n+|Mc|_n+|M(n-b)|_n+|M(n-a)|_n=Me+(Mc-8n)+(7n-Mb)+(2n-Ma)=n$, and  $\ind(S)=1$.

If $11c>5n$ and $11a>n$, let $M=9$, then $\frac{65n}{11}>13c>\frac{39n}{7}>\frac{26n}{5}>9b>\frac{39n}{8}$.
We infer that $|Mc|_n<\frac{n}{2}, |M(n-b)|_n+|M(n-a)|_n=Me+(Mc-4n)+(5n-Mb)+(2n-Ma)=3n$, and  $\ind(S)=1$.

{\bf Case 3.} $6<\frac{3n}{c}\leq7<8<\frac{3n}{b}\leq9$.

If $\frac{27}{4}<\frac{3n}{c}\leq7<8<\frac{3n}{b}\leq9$.
If $8a<n$, let $m=8$ and $k=3$, then $\ind(S)=1$. If $8a>n$, let $M=9$, then $|Me|_n+|Mc|_n+|M(n-b)|_n+|M(n-a)|_n=Me+(Mc-3n)+(4n-Mb)+(2n-Ma)=3n$, and  $\ind(S)=1$.

Next assume that $6<\frac{3n}{c}<\frac{27}{4}<7<8<\frac{3n}{b}\leq9$.

If $p_1=5$, then $n>200e$. We only need repeat the proof of Case 3 of Lemma 3.10 in [11].  Then $p_3=5$.

{\it Subcase 3.1.} $\gcd(n,7)=\gcd(n,11)=1$.

We infer that $7a>n$ and $n\geq85p_1$.

If $11b<4n$ and $11c>5n$, we have $|11e|_n + |11c|_n + |11(n-b)|_n + |11(n-a)|_n =n$ and thus $\ind(S) = 1$.

If $11b>4n$ and $11c<5n$, we have $|11e|_n + |11c|_n + |11(n-b)|_n + |11(n-a)|_n =3n$ and thus $\ind(S) = 1$.

If $11b < 4n$ and $11c < 5n$, then $\frac{n}{7}< a = c-b+e\leq \frac{5n-ep_3}{11}-\frac{n+e}{3} + e \leq\frac{4n+7e}{33}$, so $n <10e$, or $\frac{n}{7}< a = \frac{10}{9}(c-b)<\frac{40n}{297}$, either of them implies a contradiction.

If $11b > 4n$ and $11c > 5n$, then $\frac{n}{7}< a = c-b+e\leq \frac{n-e}{2}-\frac{4n+p_3e}{11} + e \leq\frac{3n+e}{22}$, so $n <7e$, or $\frac{n}{7}< a = c-b+e\leq \frac{n-p_1p_3}{2}-\frac{4n+p_1}{11} + e \leq\frac{3n+12e}{22}$, so $n<84e$, or $\frac{n}{7}< a = c-b+e\leq \frac{n-p_1p_3}{2}-\frac{4n+p_1}{11} + e \leq\frac{3n+8e}{22}$, so $n<56e$. By Lemma 5.7, each of above implies a contradiction.

{\it Subcase 3.2.} $11|n$.  We infer that $8a>n$, $e=11$, and $n\geq95e$.

The proof is similar to {\it Subcase 3.1.}

{\it Subcase 3.3.} $7|n$.  We infer that $8a>n$, $e=7$, and $n\geq145e$.

The proof is similar to {\it Subcase 3.1.}
\end{proof}

\begin{lemma}If the assumption is as in Proposition 2.16 and $k_1=2$, then $\ind(S)=1$.\end{lemma}
\begin{proof}
By Lemma 2.5, it holds that $2+t<\frac{n}{c}<\frac{n}{b}\leq3+t$ for some $t\in[0,5]$. We distinguish six cases.

{\bf Case 1.}  $t=0$. Then $2<\frac{n}{c}<\frac{n}{b}\leq3$ and $4<\frac{2n}{c}\leq5<\frac{2n}{b}\leq6$.

Similar to Lemma 5.7, we infer that $6a<n$ and hence $5|n$, by Lemma 5.8, $\ind(S)=1$.

{\bf Case 2.}  $t=1$. Then $3<\frac{n}{c}<\frac{n}{b}\leq4$. We infer that $a<\frac{n}{9}$ and $6<\frac{2n}{c}\leq7<\frac{2n}{b}\leq8$. Thus $7|n$.

If $9<\frac{3n}{c}\leq10<11<\frac{3n}{b}\leq12$, let $m\in[10,11]$(since $n$ can't be $5\times7\times11=385$) such that $\gcd(n,m)=1$. If  and $ma<n$, let $k=3$, then $\ind(S)=1$. If $ma>n$, let $M=12$, we have $Me<4a<n$, $n<Ma<2n$ and $3n<Mb<Mc<4n$, then $|Me|_n+|Mc|_n+|M(n-b)|_n+|M(n-a)|_n=Me+(Mc-3n)+(4n-Mb)+(2n-Ma)=3n$, and $\ind(S)=1$.

If $9<\frac{3n}{c}\leq10<\frac{3n}{b}\leq11$, we infer that $a<\frac{n}{12}$ and $5|n$. If $12<\frac{4n}{c}\leq13<\frac{4n}{b}\leq\frac{44}{3}$, we infer that $13a>n$. Otherwise, let $m=13$ and $k=4$, we have $\gcd(n,13)=1$(otherwise, $n=5\times7\times13=455<1000$), then $\ind(S)=1$. Let $M=22$, it is easy to check that $\gcd(n,M)=1$. If $Mc<7n$, we have $Me<n$, $n<Ma<2n$ and $6n<Mb<Mc<7n$, and $|Me|_n+|Mc|_n+|M(n-b)|_n+|M(n-a)|_n=Me+(Mc-6n)+(7n-Mb)+(2n-Ma)=3n$, thus $\ind(S)=1$. If $Mc>7n$, we have $|Me|_n<\frac{n}{2}$, $|Mc|_n<\frac{n}{2}$ and $|M(n-a)|_n<\frac{n}{2}$, then $\ind(S)=1$.

If $10<\frac{3n}{c}\leq11<\frac{3n}{b}\leq12$, we infer that $a<\frac{n}{15}$ and $11|n$. If $\frac{4n}{c}\leq 15<\frac{4n}{b}$, then $\gcd(15,n)=1$ and $15a<n$. Let $m=15$ and $k=4$, we have $\ind(S)=1$.
So $13<\frac{40}{3}<\frac{4n}{c}\leq14<\frac{4n}{b}\leq15$, and we infer that $a<\frac{n}{22}$. Let $M=25$, we have $\gcd(n,M)=1$ and $$6n+\frac{2n}{3}=\frac{100n}{15}<Mb<\frac{75n}{11}=6n+\frac{9n}{11}<7n<7n+\frac{n}{7}=\frac{50n}{7}<Mc<\frac{15n}{2}=7n+\frac{n}{2}.$$ Then $|Me|_n<\frac{n}{2}$, $|Mc|_n<\frac{n}{2}$ and $|M(n-b)|_n<\frac{n}{2}$, thus $\ind(S)=1$.

{\bf Case 3.}  $t=2$. Then $4<\frac{n}{c}<\frac{n}{b}\leq5$. We infer that $a<\frac{n}{15}$ and $8<\frac{n}{c}<9<\frac{n}{b}\leq10$. Let $m=9$ and $k=2$, then $\ind(S)=1$.

{\bf Case 4.}  $t=3$. Then $5<\frac{n}{c}<\frac{n}{b}\leq6$, $10<\frac{2n}{c}\leq 11<\frac{2n}{b}\leq12$ and $a<\frac{n}{22}$. If $\gcd(n,11)=1$, let $m=11$ and $k=2$, we have $\ind(S)=1$. If $\frac{3n}{c}<16<\frac{3n}{b}$,  let $m=16$ and $k=3$, then $\ind(S)=1$. Otherwise, we have
$a<\frac{4}{3}(a-e)=\frac{4}{3}(c-b)<\frac{4}{3}(\frac{3n}{16}-\frac{n}{6})=\frac{n}{36}$, and $16<\frac{3n}{c}\leq17<\frac{3n}{b}<18$. Similarly, $\gcd(n,17)>1$ and $27<\frac{5n}{c}<\frac{5n}{b}\leq30$.

If $27<\frac{5n}{c}\leq28<\frac{5n}{b}\leq29$, we have $a<\frac{4\times2\times5n}{3\times27\times29}$, and
$b>\frac{5n}{29}>\frac{5}{29}\frac{3\times27\times29a}{4\times2\times5}=\frac{81a}{8}>10a$, which contradicts to $s\leq9$.

If $28<\frac{5n}{c}\leq29<\frac{5n}{b}\leq30$, we have $a<\frac{4\times2\times5n}{3\times28\times30}$, and
$b>\frac{n}{6}>\frac{1}{6}\frac{3\times28\times30a}{4\times2\times5}=\frac{21a}{2}>10a$, which contradicts to $s\leq9$.

If $27<\frac{5n}{c}\leq28<29<\frac{5n}{b}\leq30$, then there exists $m$ between $\frac{5n}{c}$ and $\frac{5n}{b}$ ($m=28$ or $m=29$) such that $\gcd(n,m)=1$. Let $k=5$, we have $\ind(S)=1$.

{\bf Case 5.}  $t=4$. Then $6<\frac{n}{c}<\frac{n}{b}\leq7$. We infer that $12<\frac{2n}{c}\leq 13<\frac{2n}{b}\leq14$, $a<\frac{n}{31}$ and thus $\gcd(n,13)>1$.

If $18<\frac{3n}{c}\leq 19<20<\frac{3n}{b}\leq21$, we infer that $n=5\times13\times19$. If $\frac{5n}{c}\leq 31<\frac{5n}{b}$, let $m=31$ and $k=5$, we have $\gcd(n,m)=1$, then $\ind(S)=1$. Otherwise, $31<\frac{5n}{c}<\frac{5n}{b}\leq35$, hence $a<\frac{4}{3}\times(\frac{5n}{31}-\frac{5n}{35})=\frac{n}{40}$. Then, for any integer $m$ between $\frac{5n}{c}$ and $\frac{5n}{b}$,  we have $\gcd(n,m)=1$ and $ma<n$, hence $\ind(S)=1$.

If $18<\frac{3n}{c}\leq 19<\frac{3n}{b}\leq20$, we infer that $a<\frac{n}{45}$. If $\frac{5n}{c}<32<\frac{5n}{b}$, let $m=32$ and $k=5$, then we have done. Otherwise, $30<\frac{5n}{c}\leq\frac{95}{3}<\frac{5n}{b}<32$, thus we infer that $a<\frac{n}{72}$ and $\frac{b}{a}>\frac{5\times72}{32}=\frac{45}{4}>10$, which contradicts to $s\leq9$.

If $19<\frac{3n}{c}\leq 20<\frac{3n}{b}\leq21$, we infer that $a<\frac{n}{49}$ and $\gcd(n,20)>1$. If $27<\frac{4n}{b}$, then  $\frac{4n}{c}\leq \frac{80}{3}<27<\frac{4n}{b}$, let $m=27$ and $k=4$, we have $\ind(S)=1$. Otherwise, $\frac{76}{3}\frac{4n}{c}<\frac{4n}{b}<27$, we infer that $a<\frac{4\times5n}{19\times81}$, and $\frac{b}{a}>\frac{4}{27}\times\frac{19\times81}{20}=\frac{57}{5}>10$, which contradicts to $s\leq9$.

{\bf Case 6.}  $t=5$.  Then $7<\frac{n}{c}<\frac{n}{b}\leq8$. We infer that $14<\frac{2n}{c}\leq 15<\frac{2n}{b}\leq16$, $a<\frac{n}{42}$ and thus $\gcd(n,15)>1$.

If $21<\frac{3n}{c}\leq 22<23<\frac{3n}{b}\leq24$, we infer that $n=5\times11\times23$. If $29$ or $31$ belongs to $[\frac{4n}{c},\frac{4n}{b})$, let it be $m$ and $k=4$, we have $\gcd(n,m)=1$, then $\ind(S)=1$. Otherwise, $29<\frac{4n}{c}\leq 30<\frac{4n}{b}\leq31$, hence $a<\frac{4}{3}\times(\frac{4n}{29}-\frac{4n}{31})=\frac{32n}{3\times29\times31}$, and $\frac{b}{a}>\frac{4}{31}\times\frac{3\times29\times31}{32}=\frac{87}{8}>10$, which contradicts to $s\leq 9$.

If $21<\frac{3n}{c}\leq 22<\frac{3n}{b}\leq23$, we infer that $a<\frac{n}{60}$. If $\frac{5n}{c}<36<\frac{5n}{b}$, let $m=36$ and $k=5$, then we have done. Otherwise, $36<\frac{5n}{c}\leq37<\frac{5n}{b}\leq\frac{115}{3}$, thus we infer that $a<\frac{7n}{23\times27}$ and $\frac{b}{a}>\frac{23\times27}{7}\times\frac{3}{23}=\frac{81}{3}>10$, which contradicts to $s\leq9$.

If $22<\frac{3n}{c}\leq 23<\frac{3n}{b}\leq24$, we infer that $a<\frac{n}{66}$ and $\gcd(n,23)>1$. If $\frac{5n}{c}\leq 37$, then  $\frac{5n}{c}\leq 37<38<\frac{115}{3}<\frac{5n}{b}\leq40$. There exists $m\in[37,38]$ such that $\gcd(m,n)=1$, let $k=5$, then $\ind(S)=1$. Similarly, $\frac{5n}{c}\leq 38<39<\frac{5n}{b}\leq40$ implies $\ind(S)=1$. If $37<\frac{5n}{c}\leq 38<\frac{5n}{b}\leq39$ or $38<\frac{5n}{c}\leq 39<\frac{5n}{b}\leq40$, we infer that $b>10a$, a contradiction.
\end{proof}

{\noindent\bf Acknowledgements}

The second author is thankful to prof. Yuanlin Li and prof. Jiangtao Peng for their useful discussion and valuable comments.

\vskip30pt
\def\refname{

\end{document}